\documentclass[11pt]{amsart}
\usepackage{palatino}
\usepackage{amsmath,amssymb,amsthm,amscd, amsxtra, amsfonts,graphicx,fancybox}
\usepackage[all]{xy}
\usepackage{amssymb}
\usepackage{latexsym}
\usepackage{amscd}
\usepackage{color}
\input amssym.def
\input amssym

\newtheorem{theorem}{Theorem}[section]
\newtheorem{lemma}[theorem]{Lemma}
\newtheorem{definition}[theorem]{Definition}

\newtheorem{proposition}[theorem]{Proposition}
\newtheorem{corollary}[theorem]{Corollary}
\newtheorem{conjecture}[theorem]{Conjecture}
\newtheorem{remark}{Remark}

\def \<{\langle}
\def \>{\rangle}

\def \a{\alpha }

\newcommand{\bea}{\begin{eqnarray}}
\newcommand{\eea}{\end{eqnarray}}
\newcommand{\be}{\begin {equation}}
\newcommand{\ee}{\end{equation}}

\newcommand{\Z}{\Bbb Z}

\newcommand{\W}{\mathcal W}

\newcommand{\Zp}{{\Bbb Z}_{>0} }

\newcommand{\N}{{\Bbb Z}_{\ge 0} }
\newcommand{\C}{\Bbb C}

\newcommand{\la}{\langle}
\newcommand{\ra}{\rangle}

\usepackage{amssymb}

\begin{document}

 \title  {An explicit realization of logarithmic modules for the vertex operator algebra $\mathcal{W}_{p,p'}$}
\author{Dra\v{z}en Adamovi\'c and Antun Milas}
\address{Department of Mathematics, University of Zagreb, Croatia}
\email{adamovic@math.hr}

\address{Department of Mathematics and Statistics,
University at Albany (SUNY), Albany, NY 12222}
\email{amilas@math.albany.edu}

\begin{abstract}
By extending the methods used in our earlier work, in this  paper, we present an explicit realization of logarithmic  $\mathcal{W}_{p,p'}$--modules that have $L(0)$ nilpotent rank three. This was achieved by combining  the techniques developed in \cite{AdM-2009} with the theory of local systems of vertex operators
\cite{LL}. In addition, we also construct a new type of extension of $\mathcal{W}_{p,p'}$, denoted by $\mathcal{V}$. Our results confirm several claims in the physics literature regarding the structure of projective covers of certain irreducible representations in the principal block.  This approach can be applied to other models defined via a pair screenings.
\end{abstract}

\maketitle

\section{Introduction}

Let $p,p'$ be coprime $\geq 2$ and $c_{p,p'}=1-\frac{6(p-p')^2}{pp'}$ the minimal central charge for the Virasoro algebra.
As shown in \cite{FGST}, \cite{FGST2}, there exists  an important extension of the rational vertex operator algebra  $L(c_{p,p'},0)$, usually denoted by
 $\mathcal{W}_{p,p'}$, whose representation theory shares many similarities with representations of certain finite dimensional quantum group $\goth{g}_{p,p'}$ at roots of unity. In the same series of paper, it was predicted that these vertex algebras are in fact $\mathcal{W}$-algebras with
finitely many irreps. From the vertex algebra point of view, many claims in \cite{FGST}, \cite{FGST2} are still conjectures, including the part on the classification of irreducible representations. Also, conjecturally, $\mathcal{W}_{p,p'}$ is expected to be a $C_2$-cofinite vertex algebra, combining into an exact sequence
$$0 \longrightarrow  \overline{\mathcal{W}}  \longrightarrow \mathcal{W}_{p,p'} \longrightarrow   L(c_{p,p'},0)   \longrightarrow 0,$$
where $\overline{\mathcal{W}}$ is the maximal simple ideal in $\mathcal{W}_{p,p'}$.
The vertex algebra $\mathcal{W}_{p,p'}$ has some common features with a better understood (simple) triplet vertex algebra $\mathcal{W}_p$ studied intensively in many works (see \cite{AdM-triplet}, \cite{NT} and references therein), but there are many differences. In addition to $\mathcal{W}_{p,p'}$
not being simple,  it is also expected that (cf. \cite{FGST}, \cite{W})

\begin{itemize}

\item[(i)] The category $\mathcal{W}_{p,p'}-{\rm Mod}$ contains  $p+p'-2$ "thin" blocks, $\frac{(p-1)(p'-1)}{2}$ "thick" blocks and two semi-simple blocks.

\item[(ii)] Characters (i.e. modified graded dimensions) of irreducible $\mathcal{W}_{p,p'}$-modules  involve modular forms of weight two \cite{FGST}.

\end{itemize}

In our recent  works \cite{AdM-IMRN,AdM-2011}, motivated primarily by \cite{FGST} and  \cite{FGST2}, we started to investigate the representation theory of
vertex algebra $\mathcal{W}_{p,p'}$. There our focus was mainly on $(p,2)$--minimal models, which enjoy some special properties absent for general
$p'$. In \cite{AdM-2011}, we showed that $\mathcal{W}_{p,2}$ is $C_2$--cofinite and irrational ( for every odd $p \geq 3$). We also classified all irreducible
$\mathcal{W}_{p,2}$-modules and described the structure of the corresponding Zhu algebras. As a consequence we give evidence that $\mathcal{W}_{p,2}$ contains logarithmic modules which have $L(0)$--nilpotent rank two or three, but not higher.  We expect that the vertex algebra $\mathcal{W}_{p,p'}$ shares similar properties, but this is still an open problem. Both \cite{AdM-IMRN} and \cite{AdM-2011} deal with explicit realizations of irreducible modules, but no logarithmic $\mathcal{W}_{p,p'}$-modules were constructed there. In fact, the only rigorous result for rank $3$ logarithmic modules comes from  \cite{AdM-2011a}, where we used Zhu's algebra to prove the existence of logarithmic module of rank $3$ but only for $\mathcal{W}_{3,2}$-algebra.

Motivated by a similar circle of ideas, in \cite{AdM-2009} (and shortly afterwords independently in \cite{NT}) we used vertex operator algebra theory
to give an explicit realization of some projective covers of irreducible modules for the triplet vertex algebra $\mathcal{W}_p$
which are logarithmic  of $L(0)$--nilpotent rank two (we stress that the paper \cite{AdM-2009} deals only with special projective covers - the rest was constructed in \cite{NT}). The key idea was to realize logarithmic modules (in non-semisimple blocks) on the direct sum of two irreducible modules for the lattice vertex algebra containing the triplet vertex algebra. The methods in \cite{AdM-2009}, \cite{NT} can be also used to handle other models defined via a single screening. Thin blocks in $\mathcal{W}_{p,p'}-{\rm Mod}$ are expected to have properties similar to those for the triplet vertex algebra $\mathcal{W}_p$, while thick blocks - including  the principal block - should involve more complicated indecomposable modules. From the point of view of construction of projective covers, thin blocks  presumably can be handled by the  methods developed in \cite{AdM-2009} and \cite{NT}.

In the present work, generalizing \cite{AdM-2009},  we shall present an explicit construction of some logarithmic $\mathcal{W}_{p,p'}$-modules in
certain thick blocks which have $L(0)$--nilpotent rank three.

 Let us explain the main results and concepts of our construction.
In Section \ref{sec-def-w} we recall the definition of the vertex algebra $\mathcal{W}_{p,p'}$ as  the intersection of kernels of two commuting screening operators. This realization is useful for construction of irreducible representations, but in order to construct logarithmic representations we need to embed $\mathcal{W}_{p,p'}$  into more complicated vertex algebra. In Section  \ref{extended-1}
 we present certain results on extended vertex algebras. We construct   the extended vertex algebra
$$\mathcal{W}_{p,p'} \hookrightarrow V(p,p'),$$
 and its natural module $MV(p,p')$. The vertex algebra $V(p,p')$ is realized as a direct sum of four irreducible modules for the lattice vertex algebra $V_L$ (cf. Section \ref{extended-1}). In Section \ref{extended-2}  we construct  a non-trivial homomorphism $\Phi: \mathcal{W}_{p,p'} \rightarrow \mathcal{V}$ where $\mathcal{V}$ is a vertex algebra realized as  local fields on $V(p,p')$. In Proposition \ref{structure-extend-1}  we describe the structure of the vertex algebra $\mathcal{V}$ as a $\mathcal{W}(p,p')$--module. It is interesting to notice that the vertex algebra  $\mathcal{V}$, as a $\mathcal{W}_{p,p'}$--module, is not finitely generated, and gives a new type of extension of $\mathcal{W}_{p,p'}$. In the case $p'=2$ we are able to describe the structure of $\mathcal{V}$.  By applying the construction from \cite{AdM-2009}, in Theorem \ref{const-log-new}  we get construction of  logarithmic $\mathcal{V}$--modules. Important consequence is:

 \begin{theorem} \label{main1}
The $\mathcal{W}_{p,p'}$--modules $\overline{V(p,p')}$ and $\overline{MV(p,p')}$ have $L(0)$--nilpotent rank $3$.
\end{theorem}

When specialized to the case of $\mathcal{W}_{3,2}$-algebra, we found precise agreements with the structure of projective covers
$\mathcal{P}(1)$ and $\mathcal{P}(5)$ proposed in \cite{GRW2}. Hence

\begin{conjecture} \label{conj1} The logarithmic module $\overline{V(3,2)}$ (resp. $\overline{MV(3,2)}$) is a projective cover of  the irreducible module $\mathcal{W}(1)$ (resp.$ \mathcal{W}(5)$), where the notation $\mathcal{W}(i)$ was borrowed from \cite{AdM-IMRN}.
\end{conjecture}
Two  modules in Theorem \ref{main1} are distinguished in many aspects, and are expected to be intertwined through a simple current module.
To explain this, recall that  in \cite{AdM-IMRN} we found out that $\mathcal{W}_{p,2}$ has an important vertex superalgebra extension generated from
two primary vectors inside
\be \label{doublet}
M = \mbox{Ker}_{V_{L+\alpha /2} } Q \cap \mbox{Ker}_{ V_{ L + \alpha /2} } \widetilde{Q},
\ee
For general $p$ and $p'$, the space $M$ is a module for $\mathcal{W}_{p,p'}$ . Our next result is

\begin{theorem} \label{main2-uvod}
There is a nonzero intertwining operator of type
$$ { \overline{MV(p,p')} \choose  M \ \ \overline{V(p,p')}  } .$$
 \end{theorem}
Again, in the case of $(3,2)$-algebra the above theorem  can be made even more precise,
because  we have $\mathcal{W}(7) = M$.
\begin{corollary} Assuming Conjecture \ref{conj1}, the intertwining operator in Theorem \ref{main2-uvod} gives a non-trivial intertwining operator of type
$$ { \mathcal{P}(5) \choose  \mathcal{W}(7) \ \ \ \mathcal{P}(1) } .$$

\end{corollary}
Intertwining operators of this type were predicted in \cite{GRW2}. We also refer to \cite{W}, \cite{PR}, \cite{R1}, \cite{R2} for more about fusion
rings of general $\mathcal{W}_{p,p'}$-modules. Result in those papers are in agreement with our Theorem \ref{main2-uvod}.


Although we are still far from constructing all indecomposable projective modules as predicted in (say) \cite{W}, we believe that the methods here together with the technique of "powers" of screenings \cite{NT}, \cite{AdM-2011} will be sufficient to produce the remaining indecomposable modules and related intertwining operators.


We finish here by noting that  almost all techniques in this paper can be used to study indecomposable modules for other vertex operator
(super)algebras defined through a pair of commuting screenings.

{\bf Acknowledgment:} We thank J. Rasmussen for a correspondence.

\section{The vertex algebra $\mathcal{W}_{p,p'}$}
\label{sec-def-w}

In this part we briefly recall the construction of  $\mathcal{W}_{p,p'}$ via screening operators.

Define the rank one even lattice  $$L= {\Z} \a, \quad \la \a , \a \ra = 2 p p',
$$
where $p, p' \in {\Zp}$, $p,p' \ge 2$ and $p$ and $p'$ are relatively prime.
Let $V_L$ be the associated (rank one) lattice vertex algebra \cite{LL}.
No central extension is needed for the lattice part so we have
$$V_L=M(1) \otimes \mathbb{C}[L].$$
The vertex algebra $V_L$ is a
 subalgebra of the generalized vertex algebra $V_{
 \widetilde{L}}$, where $$\widetilde{L} = {\Z} \frac{\a}{2p p'}$$
 is the dual lattice of $L$.
Let us denote the (generalized) vertex operator map in $V_{
 \widetilde{L}}$  by $Y_{ V_{ \widetilde{L} }}$, so that
  $$Y_{ V_{ \widetilde{L} }}(a,x)=\sum_{n \in \mathbb{Q}} a_n x^{-n-1}.$$
 Define
 $$ \omega = \frac{1}{4 p p'} \a(-1) ^2 {\bf 1}  + \frac{p-p'}{2 p p'
 }\a(-2){\bf 1},$$
and $Y(\omega,z)=\sum_{n \in \mathbb{Z}} L(n)z^{-n-2}$.
It can be easily showed that $\omega$ is a conformal vector with central charge
 $$c_{p,p'}=1-\frac{6(p-p')^2}{pp'}.$$  

 We invoke the relevant screening operators \cite{FGST}, \cite{AdM-IMRN}:
\vskip 2mm
 $$ Q = e^{{\a} /{p'}}_0 = \mbox{Res}_z Y_{V_{\widetilde L}}( e^{{\a}/{p'}}, z),
 \quad \widetilde{Q} = e^{-{\a} /{p}}_0 = \mbox{Res}_z Y_{V_{\widetilde L}}( e^{- {\a}/{p}}, z).
 $$
\vskip 2mm

The screening operators $Q$ and $\widetilde{Q}$ enable us to define
certain vertex subalgebras of $V_L$. Define as in \cite{AdM-IMRN}, \cite{FGST}
\begin{center} \shadowbox{$
\mathcal{W}_{p,p'} =\mbox{Ker}_{V_L} \widetilde{Q} \cap
\mbox{Ker}_{V_L} Q.
$
}
\end{center}

\begin{remark} We still do not have a proof that a strongly generating set for $ \mathcal{W}_{p,p'}$ consists of the conformal vector $\omega$ and three primary fields of conformal weight $(2p-1) (2 p' -1)$ ( there is conjecture in \cite{FGST} about this structure). For $p'=2$ a strong generating set was found in \cite{AdM-IMRN}.
\end{remark}

Several $\mathcal{W}_{p,p'}$-modules in this paper will be {\em logarithmic}, that is they are not diagonalizable
with respect to the $L(0)$ Virasoro generator \cite{AdM-2007}, \cite{HLZ}. We also say that a $\mathcal{W}_{p,p'}$-module $M$ is of $L(0)$-nilpotent rank
$k$ if $L(0)$ admits a Jordan block of size $k$, but not higher than $k$. Equivalently, $$(L(0)-L_{ss}(0))^k=0, \quad \mbox{and}  \quad (L(0)-L_{ss}(0))^{k-1} \neq 0$$ on $M$,
where $L_{ss}$ is the semisimple part of $L(0)$.

Motivated by (conjectural) formulas of characters of irreducible $\mathcal{W}_{p,p'}$-modules \cite{FGST} we are naturally led to
\begin{conjecture} Every f.g. $\mathcal{W}_{p,p'}$-module is of $L(0)$-nilpotent rank at most $3$.
\end{conjecture}

\section{Extended vertex algebra $V(p,p')$ }\label{extended-1}

We first define semi-direct product of a vertex algebra $(V,Y)$ and its module $(M,Y_M,d)$.
Let  $$\overline{V}=V \oplus M.$$
We let
$$Y(v_1 + m_1,z)(v_2+m_2)=Y_V(v_1,z) v_2 +Y_M(v_1,z)m_2+ e^{z d} Y_M( v_2,-z)m_1, \ \ v_1,v_2 \in V, \ \ m_1,m_2 \in M.$$
The next results is known \cite{LL}, \cite{AdM-2009}.
\begin{lemma} \label{lema1} The space $\overline{V}$ has a natural vertex algebra structure.
\end{lemma}
If in addition, $V$ is a vertex operator algebra and $M$ is integrally graded, then $\overline{V}$ is vertex operator algebra.


The vertex operator algebra just described (which is certainly not simple), has potentially many indecomposable modules. The next result
gives a reasonable large source of examples.

\begin{lemma} \label{lema2} Let $\overline{V}$ be as above,  $M_2$ and $M_3$ be $V$-modules, and $\mathcal{Y}( \cdot,z)  \in I {M_3 \choose M \\ M_2}$
an intertwining operator with integral powers of $z$. Then
$$Y_{M_2 \oplus M_3} (v + m,z) (m_2+m_3)=Y(v,z)(m_2+m_3)+\mathcal{Y}(m,z)m_2,$$
$v \in V$, $m_i \in M_i$, $m \in M$,  defines an $\overline{V}$-module structure on the space $M_2 \oplus M_3$.
\end{lemma}

\begin{remark} Observe that on the same vector space $M_2 \oplus M_3$ can have several nonisomorphic $\overline{V}$-module structures.
\end{remark}

Consider now a $\overline{V}$-module $M_2 \oplus M_3$ constructed as in Lemma \ref{lema2} (possibly with the additional assumption
that the weights of $M_2$ and $M_3$ are integral).
Then we can apply again the first lemma to get a vertex (operator) algebra  structure on $$\overline{\overline{V}}=V \oplus M_1 \oplus M_2 \oplus M_3.$$
 We are also interested in $\overline{\overline{V}}$-modules. One way to do produce modules is to consider the space
 $$\overline{\overline{M}}=M_4 \oplus M_5 \oplus M_6 \oplus M_7,$$
 where $M_4 \oplus M_5$ and $M_6 \oplus M_7$ are $\overline{V}$-modules and
to employ Lemma \ref{lema2} via an intertwining operator
 \be \label{long-int}
 \mathcal{Y} \in I { M_6 \oplus M_7 \choose  M_2 \oplus M_3 \ \ M_4 \oplus M_5}
 \ee
where $\mathcal{Y}$ is viewed as intertwining operator among $\overline{V}$-modules.
It would be desirable to get an intertwining operator of this type directly from $V$-modules.
The following result comes handy.

\begin{lemma} \label{lema3} Let  $\displaystyle{\mathcal{Y}_1 \in I { M_6 \oplus M_7  \choose M_2  \ \ M_4 \oplus M_5}}$  and $\displaystyle{\mathcal{Y}_2 \in I { M_6 \oplus M_7  \choose M_3  \ \ M_4 \oplus M_5}}$ be  $V$-intertwining operators with integral powers
such that:
$$\mathcal{Y}_2( \cdot, z) (M_5)=0 \ \  {\rm and} \ \  \mathcal{Y}_2(\cdot,z) (M_4 \oplus M_5) \subset M_7((z)),$$
Then
$\mathcal{Y}=\mathcal{Y}_1 + \mathcal{Y}_2$ defines an intertwining operator of type as in (\ref{long-int}).
Consequently $\overline{\overline{M}}$ has a $\overline{\overline{V}}$-module structure.
\end{lemma}

The next goal is to apply all previous results in the setup of $V_L$-modules.
First observe that $V_L $ is a vertex operator algebra with
central charge $c_{p,p'}$ and that $V_{L- \a /p}$,  $V_{L+ \a /
p'}$ and $V_{L + \a/p' - \a /p }$ are simple  $V_L$-modules with integral weights.

Then we equip
$$V_L
\oplus V_{L-\a / p}$$
with a vertex operator algebra structure as in Lemma \ref{lema1}.
Because of a generalized vertex operator algebra structure on $(V_{\tilde{L}}, Y_{V_{\widetilde L}} ,{\bf 1})$, the restriction of $Y$
on $V_{L+\lambda_1}$ when acting on $V_{L+\lambda_2}$  gives an intertwining operator of type
 ${ V_{L+\lambda_1+\lambda_2} \choose V_{L+\lambda_1} \ V_{L+\lambda_2}}$. Moreover, it is known that this 
 space is one-dimensional. By using this fact and Lemma \ref{lema2}, we equip
$$ V_{ L + \a /p'} \oplus V_{L + \a/p' - \a /p }$$
with a
$V_{L} \oplus V_{L - \a /p}$--module.
Therefore on the space
 \begin{center}
 \shadowbox{
 $ {V(p,p') = V_L \oplus V_{L - \a /p} \oplus V_{L + \a / p'} \oplus V_{L+ \a /p' - \a /p}},
 $
 }
 \end{center}
there is a structure of a vertex operator algebra, such that the associated
vertex operators can be reconstructed from the generalized vertex
algebra $(V_{\widetilde{L}}, Y_{V_{\widetilde L}})$ as follows:
Let $u, v \in V(p,p')$,
$$ u = u_1 + u_2 + u_3 + u_4, \quad v= v_1 + v_2 + v_3 + v_4$$
$$u_1, v_1 \in V_L, \quad u_2,v_2 \in V_{L- \a/p}, \quad  u_3, v_3 \in V_{L+ \a/p'}, \quad  u_4, v_4 \in V_{L + \a/p' - \a/p}.$$
Then
\bea &&Y (u,z) v =Y_{V_{\widetilde L}} (u_1, z) v + Y_{V_{\widetilde L}} (u_2, z) (v_1 + v_3) + Y_{V_{\widetilde L}}(u_3,z) (v_1 + v_2) + Y_{V_{\widetilde L}}(u_4,z) v_1. \nonumber \eea


Now we discuss  $V(p,p')$-module.  Since we would like to invoke Lemma \ref{lema3} we are
heavily constrained with the choice of $M_i$, $4 \leq i \leq 7$. After a short analysis, besides $V(p,p')$ we narrow down
another "interesting" \footnote{We can certainly construct more examples by taking $\mathcal{Y}_1$ or $\mathcal{Y}_2$   to act  trivially.}
example.

%

\begin{proposition}  \label{another-log} The space
$$MV(p,p')=V_{L+\alpha/2} \oplus V_{L+\alpha/2-\alpha/p} \oplus V_{L+\alpha/2+\alpha/p'} \oplus V_{L+\alpha/2+\alpha/p'-\alpha/p}.$$
has a natural $V(p,p')$-module structure.
\end{proposition}


\section{Weak vertex operators and "local subsets"}

In this section we briefly outline basic facts about  "local subsets" or "local systems" following \cite{LL}, but with 
some modifications.
Let $W$ be a complex vector space, equipped with an action of a derivation operator $d$.
In \cite{LL}, a weak vertex operator (i.e. a "field") on $(W,d)$ is a formal series
$$a(z)=\sum_{n \in \mathbb{Z}} a_n z^{-n-1} \in {\rm End}(W)[[z,z^{-1}]]$$
such that $a(z)b \in \mathbb{C}((z))$ for all $b \in W$ and
\be \label{derivation}
[d,a(z)]=\frac{d}{dx}a(x).
\ee

In many constructions of vertex algebras and their modules, the operator $d$ comes naturally.
But in this paper we do not assume (\ref{derivation}). Actually, as we shall see shortly, there is no (obvious) operator $d$
satisfying (\ref{derivation}), so we only work on $W$ and not on $(W,d)$.

Denote by $\mathcal{E}(W)=Hom(W,W((z)))$ and set
$$Y_{\mathcal{E}}(a(z),x_0)b(z)=\sum_{n \in \mathbb{Z}} a(z)_n b(z) z_0^{-n-1},$$
where $n$--th product of the fields $a(z)$ and $b(z)$ is defined by
\be \label{nprod}
 a(z) _n b(z) = \mbox{Res}_{z_1} \left( ( z_1 -z) ^n a(z_1) b(z) - (-z +z_1) ^n b(z) a(z_1 ) \right).
 \ee
We should refer to the triple $(\mathcal{W},Y_{\mathcal{E}},1_{W})$ as the canonical weak vertex algebra
associated to $W$ (or $(W,d)$ if there is $d$ such that (\ref{derivation}) holds).

Weak vertex operators $a(z)$ and $b(z)$ are said to be mutually local if there exists a nonnegative integer $k$ such that
$$(z_1-z_2)^k a(z_1)b(z_2)=(z_1-z_2)^k b(z_2)a(z_1).$$
If a weak vertex operator is local with itself it is called vertex operator.

\begin{definition} A subset or a subspace $S$ of  $\mathcal{E}(W)$ is said to be local if all weak vertex operators
in $S$ are mutually local. A local subalgebra $\mathcal{E}(W)$ is a weak vertex subalgebra which is local.
\end{definition}

In fact, any local subalgebra of $\mathcal{E}(W)$ is a vertex algebra. In particular, any maximal local subspace of $\mathcal{E}(W)$ is
a vertex algebra with $W$ as a faithful module. Finally, we have a result form \cite{LL}:

\begin{theorem} Suppose $S$ is a set of mutually local vertex operator on $W$, that is a local subset of $\mathcal{E}(W)$. Then
the weak vertex subalgebra $\langle S \rangle$ generated by $S$ is a vertex algebra, with $W$ as a natural faithful $\langle S \rangle $- module,
and the set
$$ \{ u ^ {(1)} (z) _{n_1} \cdots u ^{(r)} (z) _{n_r} I_W   \ \vert \ r \in {\Zp}, \ u ^{(1)} (z), \dots, u ^{(r)}(z) \in S, \ n_1, \dots , n_r \in \Z \},$$
spans all of $\langle S \rangle$.
\end{theorem}
The module structure in the previous theorem is given by
\be \label{mod-W}
Y_{W}(a(z),x_0)=a(x_0).
\ee






%
%
%

\vskip 5mm

\section{An extension of  $\mathcal{W}_{p,p'}$}
\label{extended-2}

In this section we shall construct an extension on the vertex algebra $\mathcal{W}_{p,p'}$ on a vertex algebra of  local fields acting on a logarithmic representations constructed in \cite{AdM-2009}. 

\vskip 3mm

Consider again the vertex algebra
 $$ V(p,p') = V_L \oplus V_{L - \a /p} \oplus V_{L + \a / p'} \oplus V_{L+ \a /p' - \a /p} $$
and a family of weak vertex operators (i.e. fields)

$$ S = \{ e ^{\a /p'} (z), \widetilde{Y}(v,z) \ \vert \ v \in \mathcal{W}_{p,p'} \}, 
 $$
acting on $V(p,p')$,
where
\bea
 \widetilde{Y}(v,z) & =& Y( \Delta(e ^{-\a /p}, z) v , z) \\
\Delta(v,z) &=& z^{v_0} \exp \left( \sum_{n=1} ^{\infty}
\frac{v_n}{-n}(-z)^{-n} \right),
\eea
and
$$e^{\a/p'}(z)={Y}(e^{\a/p'},z),$$
as in Section \ref{sec-def-w}.
By using result from \cite{Li} and \cite{AdM-2009}, we can easily prove
\begin{proposition} $S$ is a local subset on $V(p,p')$.
\end{proposition}

Therefore, according to \cite{LL}, the set  $S$ generates a vertex algebra $$\mathcal{V} :=\la S \ra.$$
Let $Y_{\mathcal V}$ be the corresponding vertex operator map.
\begin{remark} \label{deri} The vertex algebra $\mathcal{V}$ is invariant under the differentiation operator $D=\frac{d}{dz}$  
such that
$$a(z) \mapsto a(z)_{-2} \cdot I(z) \quad (=Da (z)),$$
where $I(z)$ is the identity map. 

\end{remark}

Then
$$ \Phi : \mathcal{W}_{p,p'}  \longrightarrow \mathcal{V}$$
$$ v \mapsto \widetilde{Y}(v,z), \quad v \in \mathcal{W}_{p,p'}$$
is an injective vertex algebra homomorphism.   Moreover, $V(p,p')$ is a $\mathcal{V}$--module, and therefore $V(p,p')$ is a $\mathcal{W}_{p,p'}$--module.

Recall that (cf. \cite{AdM-2009})
$$ \widetilde{L(z)} = \widetilde{Y}(\omega,z) = \sum_{n \in \mathbb{Z}} \widetilde{L(n)} z^{-n-2} $$
$$=Y(\omega,z) + z^{-1} e ^{-\a/p} (z) = L(z) + z^{-1} e ^{-\a/p} (z). $$


\begin{remark} Observe that (\ref{derivation})  does not hold in general if we take $d=\widetilde{L(-1)}$ or perhaps $d=L(-1)$, and $a(z) \in S$.
\end{remark}

\begin{lemma} \label{pom-1} We have:
\bea
&& e^{\a /p'} (z) _1 \widetilde{L(z)} = e ^{\a/p'} (z) + z^{-1} e^{ \a /p' - \a /p}(z)  \\
&& (z ^{-1} e^{ \a /p' - \a /p} (z) )_1 \widetilde{L(z)} =0,  \\
&& e^{\a /p'} (z)_0 \widetilde{L(z)}= \frac{p}{p-p'} ( z^{-1} D e^{\a /p' - \a /p} (z) ),  \\
&& ( z ^{-1} e^{\a /p' - \a /p} (z) )_0 \widetilde{L(z)} = - z^{-1} (D e ^{\a /p' - \a /p} (z) ).
\eea

\end{lemma}
\begin{proof}
The proof follows easily by using (\ref{nprod}). We only prove the first formula for illustration; other formulas are proven the same way.
\bea
&& e^{\a/p'}(z) _1 \widetilde{L(z)} = \mbox{Res}_{z_1} ( z -z_1) [\widetilde{L(z)},  e^{\a/p'}(z_1) ]  \nonumber \\
&& = \mbox{Res}_{z_1} ( z -z_1) [L(z)+z^{-1}e^{-\alpha/p}(z),e^{\a/p'}(z_1)] \nonumber \\
&&=  \mbox{Res}_{z_1} z  [L(z)+z^{-1}e^{-\alpha/p}(z),e^{\a/p'}(z_1)] - \mbox{Res}_{z_1} z_1 [L(z)+e^{-\alpha/p}(z),e^{\a/p'}(z_1)] \nonumber \\
&&=\mbox{Res}_{z_1} z  [Y(e^{-\alpha/p},z),Y(e^{\a/p'},z_1)]- \mbox{Res}_{z_1} z_1 [L(z)+z^{-1} e^{-\alpha/p}(z),e^{\a/p'}(z_1)] \nonumber \\
&&= - Y(e^{\a/p'}_0 e^{-\alpha/p},z) + z^{-1} Y(e^{\a/p'}_1 e^{-\alpha/p},z)+Y(e^{\a/p'}_0 e^{-\alpha/p},z) -[L(z),e^{\a/p'}_1] \nonumber \\
&&= z^{-1} Y(e^{\a/p'-\alpha/p},z) +Y(e^{\a/p'},z) \nonumber \\
&& = z^{-1} e^{\a/p'-\alpha/p}(z) +e^{\a/p'}(z).
\eea
\end{proof}

Set $\nu_{p,p'} = \frac{p}{p-p'}. $
Define
$$ \widetilde{H(z)} = e^{\a /p'} (z) + \nu_{p,p'} z ^{-1} e^{\a /p' - \a /p} (z); $$
$$ H(z) = e^{\a /p'} (z) +  z ^{-1} e^{\a /p' - \a /p} (z). $$
Then Lemma \ref{pom-1} implies that
\bea
&& \widetilde{H(z)}_0 \widetilde{L(z)} = 0, \\
&& \widetilde{H(z)}_1 \widetilde{L(z)} = H(z), \\
&& \widetilde{H(z)}_n \widetilde{L(z)} = 0 \qquad ( n \ge 2), \\
&& \widetilde{L(z)} _1 H(z) = H(z). \\
\eea

Note also that $\widetilde{L(z)}_0 H(z) \ne  D H(z)$.

By using commutator formulae in the vertex algebra $\mathcal{V} $ we get
\bea \label{com-field}[ \widetilde{L(z)}_{n+1} , \widetilde{H(z)}_m ] = H(z)_{m+n}. \eea

The vertex algebra $\mathcal{V}$ is an extension of $\mathcal{W}_{p,p'}$. Let us give a description of $\mathcal{V}$ as an $\mathcal{W}_{p,p'}$--module.

Let  $\mathcal{I}$  be the ideal in $\mathcal{V}$ generated by $\mathcal{W}_{p,p'}. z^{-1} e^{ \alpha /p' - \alpha / p} (z). $

\begin{proposition} \label{structure-extend-1}
As a $\mathcal{W}_{p,p'}$--module,   $\mathcal{I}$  is isomorphic to a direct sum of infinitely many cyclic modules generated by singular vectors
 $\{ z^{-k} e^{ \alpha /p' - \alpha / p} (z), \quad  k \in {\Zp} \}$, i.e.,
 $$ \mathcal{I} = \bigoplus_{k =1} ^{\infty} \mathcal{I}_k, \quad \mathcal{I}_k = \mathcal{W}_{p,p'} .  z^{-k} e^{ \alpha /p' - \alpha / p} (z). $$
\end{proposition}
\begin{proof}
The operator $U=L(z)_0-D$ commutes with the $\mathcal{W}_{p,p'}$-action, so it defines an intertwining operator between $\mathcal{I}$ and itself.
Notice first
$$ \widetilde{L(z)}_0  ( z^{-1} e^{ \alpha /p' - \alpha / p} (z) ) =  z ^{-1} D e^{ \alpha /p' - \alpha / p} (z)  \in \mathcal{I}. $$
Therefore
$$ z^{-2}  e^{ \alpha /p' - \alpha / p} (z) = - D (z^{-1} e^{ \alpha /p' - \alpha / p} (z) ) + \widetilde{L(z)}_0  ( z^{-1} e^{ \alpha /p' - \alpha / p} (z) ) \in \mathcal{I} $$
$$=U (z^{-1} e^{ \alpha /p' - \alpha / p} (z)),$$
where we used Remark \ref{deri}.
So all singular vectors are generated by application of $U^k$ on

\noindent $z^{-1} e^{ \alpha /p' - \alpha / p} (z)$.
\end{proof}
\vskip 3mm
In the case $p'=2$,  by using results and methods developed in \cite{AdM-2011} one proves the following result:
\begin{proposition}
Assume that $p'=2$. Then for every $k \in {\Zp}$  we have $\mathcal{I}_k \cong \mbox{Ker}_{ V_{L + \a / 2-\a /p} } Q $.
Moreover, $${\mathcal V} / \mathcal{I} \cong \mathcal{W}_{p,2} \oplus \mbox{Ker}_{ V_{ L + \alpha /2} } Q. $$
\end{proposition}

\section{ Logarithmic $\mathcal{W}_{p,p'}$-modules}
\label{const-log}

In this section we shall construct certain logarithmic modules for  $\mathcal{W}_{p,p'}$. Our main technique will be a deformation of the vertex algebra $\mathcal{V}$ constructed in Section \ref{extended-2}.

\vskip 3mm

We start with the following simple result.

\begin{proposition} \label{vir-new}
Assume that $V$ is a vertex algebra with conformal vector $\omega$ and Virasoro field $L(z) = Y(\omega,z)$. Assume that $b \in V$ such that
$$L(n) b  = \delta_{n,0} b \ \ (n \in {\N}), \quad [b_i,b_j]=0 \ \ (i,j \in \Z). $$
Then  $$\overline{L(z)} = Y(\omega, z) + z^{-1} Y(b,z) $$
is a Virasoro field acting on any $V$--module.
\end{proposition}
\begin{proof}
Let $\overline{L}(n) = L(n) + b_n$. We have:
\bea [\overline{L(n)}, \overline{L(m)} ] &=& [L(n) + b_n , L(m) + b_m] \nonumber \\
& = & [L(n), L(m)] + (n-m) (L(0) b)_{n+m} - (L(-1) b )_{n+m+1} +  (L(-1) b )_{n+m+1} \nonumber \\
& = &  (n-m) ( L(n+m) + b_{n+m}) + \frac{n^3 - n}{12} c \delta_{n+m,0}. \nonumber
\eea
The proof follows.
\end{proof}
\begin{remark}
Note that we do not require  $L(-1) b =  D b$, where $D$ is the canonical derivation on $V$. 
\end{remark}
\begin{corollary} Define
$$\overline{L(z)} = \widetilde{L(z)} + z^{-1} H(z).$$
Then $\overline{L(z)}$  is a Virasoro field acting on $V(p,p')$ and every $V(p,p')$--module.
\end{corollary}
\begin{proof}
 Inside the  vertex algebra $\mathcal{V}$, we have  conformal vector $\omega= \widetilde{L(z)}$  and vector $b= H(z)$, and  the associated vertex operators
 $Y_{\mathcal V} (\widetilde{L(z)}, z_0) = \sum_{n \in \Z} \widetilde{L(z)} _n z_0 ^{-n-1}$,  $Y_{\mathcal V} (H(z), z_0) = \sum_{n \in \Z} H(z) _n z_0 ^{-n-1}$ acting on $\mathcal{V}$. We apply Proposition \ref{vir-new} in this case and get the Virasoro field
 $ Y_{\mathcal V} (\widetilde{L(z)}, z_0) + z_{0} ^{-1} Y_{\mathcal V} (H(z), z_0)$. Applying this field on $V(p,p')$-modules we get the Virasoro field $\overline{L(z)}$, as required.
\end{proof}

\begin{lemma}
$\widetilde{H(z)}_0 = 0 $ on  ${\mathcal V}$.
\end{lemma}
\begin{proof}
As in Lemma 5.1 of \cite{AdM-2009} we see that
$$ [\widetilde{H(z)}_0 , \widetilde{a_n} ] = \widetilde{(Q a)_n}, \qquad (a \in \mathcal{W}_{p,p'} ). $$
Since $Q a = 0$, we conclude that
$$ [\widetilde{H(z)}_0 , \widetilde{Y}( a,z) ] = 0, \qquad (a \in \mathcal{W}_{p,p'} ). $$
This easily implies the proof.
\end{proof}

Therefore $\widetilde{H(z)}_n$, $n \in \Z$, define on $\mathcal{V}$ the structure of a module for the Heisenberg algebra such that $\widetilde{H(z)}_0$ acts trivially.
Therefore the field
$$ \Delta(\widetilde{H(z)},z_1) = z_1 ^ {\widetilde{H(z)}_0} \exp  \left(\sum_{n = 1} \frac{\widetilde{H(z)} _{n} }{-n} z_1 ^{-n}  \right) $$ is well defined on $\mathcal{V}$. As in \cite{AdM-2009} (cf. \cite{Li}, \cite{H}; see also \cite{FFHST}) we have the following result:

\begin{theorem} \label{const-log-new}
\item[(1)]For every $ v(z) \in {\mathcal V}$ we define
$$ \overline{Y} (v(z),z_1) = Y_{\mathcal V} (\Delta(\widetilde{H(z)},z_1) v(z),z_1). $$
Then
$({\mathcal V}, \overline{Y})$ is a ${\mathcal V}$--module.

\item[(2)] Assume that $(M,Y_M(\cdot,z_1)$ is a weak ${\mathcal V}$--module. Define the pair $(\overline{M}, \overline{Y}_{\overline{M}}(\cdot, z_1) )$ such that
\bea &&  \overline{M}= M \qquad \mbox{ as a vector space}, \nonumber \\
&& \overline{Y}_{\overline M} (v(z),z_1) = Y_{ M} (\Delta(\widetilde{H(z)},z_1) v(z),z_1). \nonumber \eea

Then $(\overline {M}, \overline{Y}_{\overline{M}}(\cdot, z_1) )$
is a weak ${\mathcal V}$--module. In particular, $(\overline {M}, \overline{Y}_{\overline{M}}(\cdot, z_1) )$ is a $\mathcal{W}_{p,p'}$--module.
\end{theorem}

Recall that $V(p,p')$ is a module for the vertex algebra ${\mathcal V}$ with the vertex operator map
$$ Y (v(z), z_0 ) = v(z_0), \qquad v(z) \in {\mathcal V}.$$
Applying the above construction we get a new explicit  realization of logarithmic modules for $\mathcal{W}_{p,p'}$ of $L(0)$ nilpotent  rank $3$.
 \begin{theorem}
 $ (\overline{V(p,p')}, \overline{Y})$ is a $\mathcal{W}_{p,p'}$--module  such that
 $$\overline{Y} (\omega, z) = \overline{L}(z).$$
 Operator $\overline{L(0)}$ acts on $\overline{V(p,p')}$ as
 \begin{center} \shadowbox{
 $ \overline{L(0)} = L(0) + Q + \widetilde{Q} + e ^{\a /p' - \a /p}_{-1}$
 }
 \end{center}
 and it has $L(0)$-nilpotent rank $3$.
 \end{theorem}

\begin{proof}
To prove the formula for $\overline{L(0)}$,
we simply use
$$ \overline{Y}_{\overline M} (v(z),z_1) = Y_{ M} (\Delta(\widetilde{H(z)},z_1) v(z),z_1),$$
and apply the formulas $\widetilde{H(z)}_0 \widetilde{L(z)} = 0,$ $\widetilde{H(z)}_1 \widetilde{L(z)} = H(z)$, and
$\widetilde{H(z)}_n \widetilde{L(z)} = 0$, $ n \ge 2$ proven earlier.

From the definition of $ \overline{V(p,p')}$ and of $\omega$, we easily see that
the (generalized) eigenspace  $\overline{V(p,p')}_0$ is $2$-dimensional spanned by ${\bf 1}$ and $e ^{\a /p - \a /p'}$ so Jordan block
of length $3$ does not appear on the lowest weight subspace.
Since $L(0)$ acts as identity on the top component of $V_{L-\alpha/p}$, and the same way on top subspace of $V_{L+\alpha/p'}$, we first analyze
$\overline{V(p,p')}_1$, which is four-dimensional spanned by $\alpha(-1){\bf 1}$, $e^{-\alpha/p}$, $e^{\alpha/p'}$ and $e ^{\a /p' - \a /p}_{-1}{\bf 1}$.
We compute
\bea
&& \overline{L(0)} \cdot \alpha(-1){\bf 1}=\alpha(-1){\bf 1}+(Q+\tilde{Q})\alpha(-1){\bf 1}+ e ^{\a /p - \a /p'}_{-1} \alpha(-1) {\bf 1}, \nonumber \\
&& = \alpha(-1){\bf 1}+2p' e^{-\alpha/p}-2p e^{\alpha/p'}+\alpha(2p-2p')(\frac{1}{p}-\frac{1}{p'})\alpha(-1)e ^{\a /p - \a /p'}, \nonumber \\
&& \overline{L(0)} \cdot e^{-\a/p}=e^{-\a/p}-\frac{\alpha(-1)}{p'} e ^{\a /p - \a /p'}, \nonumber \\
&& \overline{L(0)} \cdot e^{\a/p'}=e^{\a/p'}+\frac{\alpha(-1)}{p} e ^{\a /p - \a /p'}, \nonumber \\
&& \overline{L(0)} \cdot \alpha(-1) e ^{\a /p - \a /p'}=\alpha(-1) e ^{\a /p - \a /p'}. \nonumber
\eea
Observe that $L(0)=\overline{L(0)}_{ss}$. Consequently,
$$(\overline{L(0)}-{L(0)})^2  \alpha (-1) =-4 \alpha(-1) e ^{\a /p - \a /p'} \neq 0.$$
It is easy to see that $(L(0)-\overline{L(0)})^3=0$. But  $L(0)$ is the semisimple part of
$\overline{L(0)}$, so the assertion follows.
\end{proof}

In Proposition \ref{another-log} we constructed a $V(p,p')$-module, which we denoted by $MV(p,p')$. By using the above method we obtain
a $\mathcal{W}_{p,p'}$-module structure on $\overline{MV(p,p')}$.

\begin{theorem} \label{main2} The module $\overline{MV(p,p')}$ is of $L(0)$-nilpotent rank $3$.
\end{theorem}
\begin{proof}
By using the theory of Feigin-Fuchs modules (cf. \cite{FF}, \cite{FGST2}, \cite{IK}) we see that there is a subsingular vector $w$ in
$M(1) \otimes e^{\alpha /2}$ such that
$$ Q   w = e ^{ \alpha /2 + \alpha /p'}, \quad L(0) w = \frac{(p+2)(p'+2)}{4} w. $$
Then
$$(\overline{L(0)}-{L(0)})^2 w = 2 \   \widetilde{Q} Q  w = 2 e ^{-\alpha /p} _0 e ^{\alpha /2 + \alpha /p'} \ne 0. $$
The proof follows.
\end{proof}

\section{Construction of  certain intertwining operators}
\label{const-int-sect}

So far  we actually  constructed intertwining operators of types
$$ { \overline{V(p,p')}   \choose \mathcal{V} \ \ \overline{V(p,p') } } \qquad \mbox{and} \qquad { \overline{MV(p,p')}  \choose \mathcal{V} \ \  \overline{MV(p,p')} } . $$
Now we shall construct  intertwining operators acting between $\overline{V(p,p') }$ and   $\overline{MV(p,p')}$.
%

 The vertex algebra $\mathcal{V}$ can be treated as a vertex algebra of fields acting on $ V(p,p') \oplus MV(p,p')$.  Let $\mathcal{W}$ be the maximal local subspace of fields acting on $V(p,p') \bigoplus MV(p,p')$ containing $\mathcal{V}$.
Then $\mathcal{W}$ is a vertex algebra which can be treated as a $\mathcal{V}$--module.

We now take intertwining operator $\mathcal{Y}$ of type
$${ MV(p,p') \ \choose MV(p,p') \ \ \ V(p,p') }$$
in the category of $V(p,p')$--modules (existence of such intertwining operator easily follows from  the considerations  in Section \ref{extended-1}).

Consider the  $\mathcal{W}_{p,p'}$--module

$$M = \mbox{Ker}_{V_{L+\alpha /2} } Q \cap \mbox{Ker}_{ V_{ L + \alpha /2} } \widetilde{Q} \subset V_{ L + \alpha /2}. $$

Applying the operator $\Delta(e ^{-\alpha /p},z)$ we get the intertwining operator
$$ \widetilde{\mathcal{Y}} (v,z) = \mathcal{Y}(\Delta(e ^{-\alpha /p},z) v,z) $$
of type
$$ { MV(p,p') \ \choose M \ \ \ V(p,p') } $$
in the category of $\mathcal{W}_{p,p'}$--module.

Next, we assume that   $\widetilde{\mathcal{Y}} (v,z) $ acts trivially on $MV(p,p')$, so $\widetilde{\mathcal{Y}} (v,z) $ can be considered as field on $V(p,p') \bigoplus MV(p,p')$.  One can see the following important lemma.

\begin{lemma} The fields
$$\{  \widetilde{\mathcal{Y}} (v,z) ,  \ v \in M \} \cup  S$$
are mutually local on $V(p,p') \bigoplus MV(p,p')$.
\end{lemma}
Therefore, the fields $\widetilde{\mathcal{Y}} (v,z)$, $ v \in M$,   belong to the vertex algebra $\mathcal{W}$.  They generate the following
   $\mathcal{V}$--module:
$$ \mathcal{MV} := \mathcal{V} . M = \mbox{span}_{\C} \{ a(z) _n   \widetilde{\mathcal{Y}} (v,z) \ \vert \ a(z) \in {\mathcal V}, \ v \in M,  n \in {\Z} \}. $$
Define the intertwining operator $ \mathcal{Y}_1 (\cdot, z)$ by
$$ \mathcal{Y}_1 ( u(z), z_0) = u(z_0), \quad u(z) \in \mathcal{MV}.$$
One can easily see that
$\mathcal{Y}_1(\cdot, z)$ is an intertwining operator of type
$$ { MV(p,p') \choose \mathcal{MV} \ \ V(p,p') }$$
in the category of $\mathcal{V}$--modules.

   Applying the operator $\Delta(\widetilde{H(z)},z_0)$, we get intertwining operator
 $$ \widetilde{\mathcal Y}_1(u(z), z_0) = {\mathcal Y}_1( \Delta(\widetilde{H(z)}, z_0)  u(z), z_0)$$
of type
 \bea  && { \overline{MV(p,p')} \choose \mathcal{MV} \ \ \overline{V(p,p')} }. \label{int-v-1} \eea
 (Note that $\widetilde{H(z)}_0$ acts trivially on $\mathcal{MV}$.)

By considering (\ref{int-v-1}) as intertwining operator in the category of $\mathcal{W}_{p,p'}$--modules, we get:

\begin{theorem}  There is an non-zero intertwining operator  of type
$$ { \overline{MV(p,p')} \choose M  \ \ \overline{V(p,p')} }$$
in the category of $\mathcal{W}_{p,p'}$--modules.
\end{theorem}
\begin{remark}
In the case $p'=2$, $M$ is isomorphic to the irreducible $\mathcal{W}_{p,2}$--module with lowest weight $h'= 3 p-2$ (see \cite{AdM-2011} for details).
\end{remark}

\section{$\mathcal{W}_{3,2}$-algebra}

The vertex algebra $\mathcal{W}_{3,2}$ has attracted considerable interest in the physics literature, primarily because its central charge is zero.
The main ingredients of $\mathcal{W}_{3,2}$ representation theory were worked out in \cite{AdM-IMRN}.
In \cite{GRW1} \cite{GRW2}, among many other things, M. Gaberdiel, I. Runkel and  S. Wood investigated indecomposable projective $\mathcal{W}_{3,2}$-modules by using various tools coming from tensor categories. According to \cite{GRW2}, based on known properties of projective modules on the quantum group side \cite{FGST2},  the internal structure of  projective cover $\mathcal{P}(1)$ should exhibit embedding structure as on the figure
 $$
 
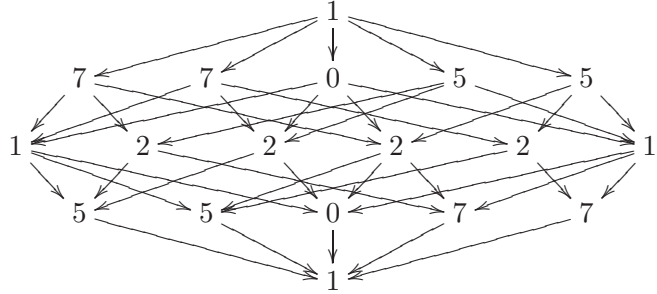
\xymatrix@-1pc{    & &  &  &  & \ar[d] \ar[dll] \ar[dllll] 1 \ar[drr] \ar[drrrr] &  &  & & &  \\
 & 7 \ar[dl] \ar[dr] \ar[drrrrr] & & 7 \ar[dlll] \ar[dr] \ar[drrrrr]  & & 0 \ar[dlllll] \ar[drrrrr] \ar[dr] \ar[dl] & & 5 \ar[dlllll] \ar[dlll] \ar[drrr] & & \ar[dlll] \ar[dl] \ar[dr] 5 &  \\
 1 \ar[dr]  \ar[drrr]  \ar[drrrrr] & & 2 \ar[dl] \ar[drrrrr] &  &  \ar[dr] \ar[dlll] 2 & & \ar[dlll]2 \ar[dl] \ar[dr] & & 2 \ar[dlllll] \ar[dr] & & 1\ar[dlllll] \ar[dlll] \ar[dl]  \\
 & 5 \ar[drrrr] & & 5 \ar[drr]  & & 0 \ar[d]   & & 7 \ar[dll] & & 7 \ar[dllll] & \\
 & & & & & 1 & & & &  &
}
$$
where we abbreviated $i$ for the irreducible module $\mathcal{W}(i)$ (we are using the notation
from \cite{AdM-IMRN} here).  The arrows indicate the standard way of representing embeddings and quotients (for example, the module with only incoming arrows forms the socle of $\mathcal{P}(1)$).

Now, we compare this diagram with the logarithmic module $\overline{V(3,2)}$ constructed in the previous section on the direct sum of four irreducible
$V_L$-modules $V_L \oplus V_{L-\alpha/3} \oplus V_{L+\alpha/2} \oplus V_{L+\alpha/6}$. By using classification of irreps in \cite{AdM-IMRN}, and the embedding diagrams of Feigin-Fuchs modules we easily find embedding structure for $V_{L}$ and $V_{L+\alpha/6}$
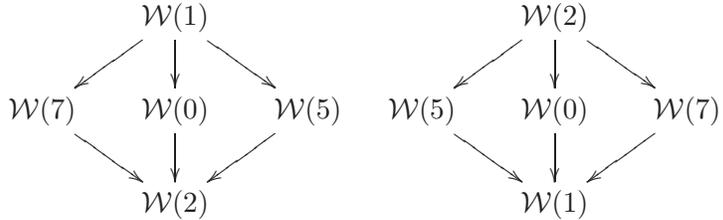
\begin{center}
$ \xymatrix@-0.5pc{     & \mathcal{W}(1) \ar[dl] \ar[d] \ar[dr] &  \\ \mathcal{W}(7) \ar[dr] & \mathcal{W}(0) \ar[d] & \mathcal{W}(5) \ar[dl] \\ & \mathcal{W}(2) & } \ \ \ $
$ \xymatrix@-0.5pc{     & \mathcal{W}(2) \ar[dl] \ar[d] \ar[dr] &  \\ \mathcal{W}(5) \ar[dr] & \mathcal{W}(0) \ar[d] & \mathcal{W}(7) \ar[dl] \\ & \mathcal{W}(1) & }$
\end{center}
The other two $V_L$-modules have embedding structure:
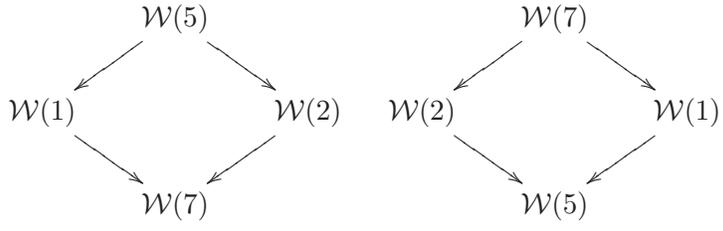
\begin{center}
$ \xymatrix@-0.5pc{     & \mathcal{W}(5) \ar[dl]  \ar[dr] &  \\ \mathcal{W}(1) \ar[dr] & & \mathcal{W}(2) \ar[dl] \\ & \mathcal{W}(7) & } \ \ \ $
$ \xymatrix@-0.5pc{     & \mathcal{W}(7) \ar[dl]  \ar[dr] &  \\ \mathcal{W}(2) \ar[dr] &  & \mathcal{W}(1) \ar[dl] \\ & \mathcal{W}(5) & }$
\end{center}
Now it is clear that these four diagrams, when put together, and by adding appropriate arrows coming from the deformed structure, should give the module $\mathcal{P}(1)$.
So our work is in perfect agreement with \cite{GRW2} and \cite{FGST2}.

Next, we analyze the logarithmic module $\overline{MV(3,2)}$ from Theorem \ref{main2}. It is interesting to observe that  $\overline{MV(3,2)}$ is {\em identical} to $\overline{V(3,2)}$ as  vector spaces. Yet,

\begin{proposition} $\overline{MV(3,2)} \ncong  \overline{V(3,2)}$
\end{proposition}

The proof of the proposition follows immediately by looking at the {\em socle} of the module in question. For instance $\mathcal{W}(5)$ belongs in the socle
of $\overline{MV(3,2)}$, but not in the socle of $\overline{V(3,2)}$. Again, if we put four diagrams above diagrams to form $\overline{MV(3,2)}$, then - modulo some
arrows -  the internal structure of a (hypothetical) projective cover $\mathcal{W}(5)$ from \cite{GRW2}:
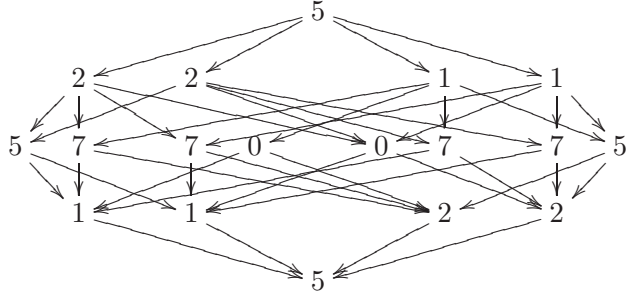
 $$
 \xymatrix@-1pc{    & &  &  &  &  \ar[dll] \ar[dllll] 5 \ar[drr] \ar[drrrr] &  &  & &  \\
 & 2 \ar[d] \ar[dl] \ar[drr] \ar[drrrrr] & & 2 \ar[dlll] \ar[drrr] \ar[drrrr] \ar[drrrrrr]  & &  & & 1 \ar[dllllll] \ar[dlll] \ar[drrr] \ar[d]  & & \ar[dlll] \ar[dllllll]  \ar[d] 1 \ar[dr] &  \\
 5 \ar[dr]  \ar[drrr]  & 7 \ar[d] \ar[drrrrrr] &  & 7 \ar[d] \ar[drrrr] &  \ar[dlll]  \ar[drrr] 0  &   & 0 \ar[dlll] \ar[drrr]  & 7 \ar[dllllll] \ar[drr] & & 7 \ar[d] \ar[dllllll] &  5 \ar[dlll]  \ar[dl]  \\
 & 1 \ar[drrrr] & & 1 \ar[drr]  & &  & & 2 \ar[dll] &  &  2 \ar[dllll]  & \\
 & & & & & 5 & & & &
}
$$
where it  was denoted by $\mathcal{P}(5)$.

\begin{remark}
{\em It is very likely that  a combination of methods from \cite{NT}, \cite{AdM-2012} based on "powers" of screening operators  can be used to construct other two projective modules $\mathcal{P}(2)$ and $\mathcal{P}(7)$ on the {\em same} direct product of four irreducible $V_L$-modules. The remaining projective cover in the principal block
, namely $\mathcal{P}(0)$, seems to be the hardest to construct explicitly (although its structure was conjectured in \cite{GRW2}). 
We do not pursue this direction in the present paper.}
\end{remark}

From everything being said we end with a hypothesis.
\begin{conjecture}  Modules $\overline{V(p,p')}$ and $\overline{MV(p,p')}$ above are projective covers of particular irreducible $\mathcal{W}_{p,p'}$-modules of lowest conformal weight $1$ and $\frac{(p+2)(p'+2) }{4}$, respectively.
\end{conjecture}

Even in the $c=0$ case we cannot prove projectivity of either of the modules because we lack information about other indecomposable modules in the
block and the related ${\rm Ext}$-groups.

\begin{remark}
In Appendix A1 of \cite{GRW2}, the authors studied  intertwiners  between projective $\mathcal{W}(2,3)$--modules $\mathcal{P}(1)$ and $\mathcal{P}(5)$.
In the language of logarithmic tensor product \cite{HLZ},  relation between these modules should be interpreted as
\bea && \mathcal{P}(5) = \mathcal{P}(1) \widehat{ \otimes } \mathcal{W}(7). \label{tensor-7} \eea

In Section \ref{const-int-sect}
 we have constructed
an intertwining operator of type
$${ \mathcal{P}(5) \choose \mathcal{W}(7) \ \  \mathcal{P}(1) }.$$
The existence of this intertwining operator is in agreement with (\ref{tensor-7}).

 We also
get some intertwining operators of the form
$$ { \mathcal{P}(1)  \choose \mathcal{W}(h) \ \  \mathcal{P}(1) }  \qquad  \mbox{and} \qquad{ \mathcal{P}(5)  \choose \mathcal{W}(h) \ \  \mathcal{P}(5) }. $$

Existence of such intertwining operators was predicted by the fusion rules analysis from  in \cite{GRW1},  \cite{GRW2}, \cite{W}.
\end{remark}

\section{Generalized twisted modules and logarithmic intertwining operators}

Let us finish with a  brief comment on logarithmic intertwining operators among $\mathcal{W}_{p,p'}$-modules and
generalized twisted modules associated to automorphisms of infinite order. In \cite{AdM-2009}, Theorem 9.1, we already constructed
examples of logarithmic intertwining operators \cite{AdM-2007}, \cite{HLZ} among triples of logarithmic $\mathcal{W}_{p,p'}$ modules, involving
at most linear logarithmic factor ${\rm log}(z)$. Shortly after, Huang in \cite{H}, Theorem 5.8, connected our construction with his notion of generalized twisted modules, and based on \cite{AdM-2007} provided examples coming from automorphisms $e^{2 \pi i Q}$ and $e^{2 \pi i \tilde{Q}}$.

Results from the previous section raise the issue of construction of more general
intertwining operators and related generalized twisted $e^{2 \pi i Q}$ and $e^{2 \pi i \tilde{Q}}$-modules,
 which in addition to linear logarithmic terms also involve the quadratic term ${\rm log}^2(z)$. This of course is supposed to capture the ubiquitous
 rank $3$ nilpotency in $\mathcal{W}_{p,p'}$-Mod. We should say that for the triplet vertex algebra $\mathcal{W}_p$, due to rank $2$ nilpotency, linear logarithmic intertwiners are sufficient.

Because our approach relies heavily on local systems, we propose to study a certain logarithmic extension of $\mathcal{V}$ denoted by $\mathcal{V}_{log}$
which can be viewed as a logarithmic $\mathcal{V}$-module in the weakest sense.  This kind of logarithmic extension will be further deformed
 via $\Delta(\widetilde{H(z)},z_1)$ to construct desired intertwining operators and generalized twisted modules.
This and similar constructions will be pursued in the forthcoming publication.


\end{document}